\documentclass[12pt,a4paper]{amsart}

\usepackage{amsfonts}
\usepackage{amsthm}
\usepackage{amsmath}
\usepackage{amscd}
\usepackage[latin2]{inputenc}
\usepackage{t1enc}
\usepackage[mathscr]{eucal}
\usepackage{indentfirst}
\usepackage{graphicx}
\usepackage{graphics}
\usepackage{pict2e}
\usepackage{epic}
\numberwithin{equation}{section}
\usepackage[margin=2.4cm]{geometry}
\usepackage{epstopdf} 

\usepackage{multicol, amstext, amsgen, amsbsy, amsopn, amssymb, subfigure}
\usepackage{longtable, tabu}

\newcommand{\tens}[1]{%
  \mathbin{\mathop{\otimes}\displaylimits_{#1}}%
}

\theoremstyle{plain}
\newtheorem{Th}{Theorem}[section]
\newtheorem{Cor}[Th]{Corollary}
\newtheorem{Prop}[Th]{Proposition}
\newtheorem{Lemma}[Th]{Lemma}

 \theoremstyle{definition}
\newtheorem{Def}[Th]{Definition}

\begin{document}
\title{Lech's Inequality for the Buchsbaum-Rim Multiplicity and Mixed Multiplicity}

\author{Vinh Nguyen and Kelsey Walters}

\address{Vinh Nguyen \\ Department of Mathematics \\ Purdue University \\
West Lafayette \\ IN 47907 \\ USA} 
\email{nguye229@purdue.edu}

\address{Kelsey Walters \\ Department of Mathematics \\ Purdue University \\
West Lafayette \\ IN 47907 \\ USA} 
\email{kjlwalters@purdue.edu}



\begin{abstract} We generalize an improved Lech bound, due to Huneke, Smirnov, and Validashti, from the Hilbert-Samuel multiplicity to the Buchsbaum-Rim multiplicity and mixed multiplicity. We reduce the problem to the graded case and then to the polynomial ring case. There we use complete reductions, studied by Rees, to prove sharper bounds for the mixed multiplicity in low dimensions before proving the general case.
\end{abstract}
\maketitle
\section{Introduction} For a Noetherian local ring $(R,m)$ of dimension $d$ and an $m$-primary ideal $I$, Lech \cite[{Theorem }3]{LechIneq} proved the following bound on the Hilbert-Samuel multiplicity, $e(I) \leq d!\lambda(R/I)e(R)$, where $e(I)$ denotes the Hilbert-Samuel multiplicity of $I$, $e(R)$ denotes the Hilbert-Samuel multiplicity of $m$, and $\lambda(R/I)$ denotes the length of $R/I$. Recently, Huneke, Smirnov, and Validashti \cite[{Theorem }6.1]{HSV-Lech} improved the bound to $e(mI) \leq d!\lambda(R/I)e(R)$ for $d \geq 4$. 

There has been significant work to generalize the Hilbert-Samuel multiplicity. Buchsbaum and Rim generalized the Hilbert-Samuel multiplicity to submodules $E \subseteq F = R^r$  of finite-rank free modules such that $\lambda(F/E) < \infty$. Let $\textrm{Sym}(E)$ denote the symmetric algebra of a finite module $E$. Let $E^n$ denote the $n$-th degree component of the image of the natural map $\textrm{Sym}(E) \to \textrm{Sym}(F)$, and let $F^n$ denote the $n$-th degree component of $\textrm{Sym}(F)$. Buchsbaum and Rim proved that for $n >> 0$ the lengths $\lambda(F^n/E^n)$ is a polynomial function in $n$ \cite[Theorem 3.1]{Buchsbaum-Rim}. The normalized leading coefficient of this polynomial is called the Buchsbaum-Rim multiplicity of $E$, and is denoted $br(E)$. It plays a key role in the theory of equisingularities in complex-analytic geometry, see \cite{gaffney-kleiman}. 

Another way to generalize the Hilbert-Samuel multiplicity is to measure the colengths of multiple $m$-primary ideals $I_1,...,I_r$ . For $n_1,...,n_r>>0$, $\lambda(R/(I_1^{n_1}\cdots I_r^{n_r}))$ is a polynomial of degree $d$ in $r$ variables, $x_1,...,x_r$ \cite[Theorem 17.4.2]{SwansonHuneke}. The resulting polynomial has a family of leading coefficients; normalized they are called the mixed multiplicities of $I_1,...,I_r$. For non-negative integers $a_1,...,a_r$ with $a_1+...+a_r=d$, if we list each ideal $I_i$, $a_i$ number of times as $I_1,...,I_d$, then $a_1!\cdots a_r!$ times the coefficient of $x_1^{a_1}\cdots x_r^{a_r}$ is called the mixed multiplicity of $I_1,...,I_r$ of type $(a_1,...,a_r)$. It is denoted as $e(I_1,...,I_d)$. We will define the Buchsbaum-Rim multiplicity and mixed multiplicities in more detail in section 2.

In this paper we generalize the improved Lech Bound to the Buchsbaum-Rim multiplicity and also to the mixed multiplicity. Since the Buchsbaum-Rim multiplicity generalizes the Hilbert-Samuel multiplicity, our result generalizes that of Huneke, Smirnov, and Validashti \cite[{Theorem }6.1]{HSV-Lech}.

\begin{Th} \label{main} Let $(R,m)$ be a Noetherian local ring with dim $R = d \geq 4$, and $E \subseteq F = R^r$ a submodule of a finite-rank free module with $\lambda(F/E) < \infty$ and $E \subseteq mF$. Then
$$\belowdisplayskip=20pt br(mE) < \frac{(d+r-1)!}{r!}\lambda(F/E)e(R).$$
\end{Th}

\begin{Th}\label{main-mixed}
Let $(R,m)$ be a Noetherian local ring with dim $R = d \geq 4$, and let $I_1,..,I_d$ be $m$-primary ideals. Then

$$e(mI_1,...,mI_d) < (d-1)! \sum_{i=1}^{d}\lambda(R/I_i)e(R).$$
\end{Th}

We prove the theorems in several steps, using the techniques of Lech in \cite{LechIneq} and Huneke, Smirnov, and Validashti in \cite{HSV-Lech}. First we pass to the associated graded ring to reduce to the case where $R$ is a standard graded ring over $R/m = k$, $E$ is a homogeneous submodule, and each $I_i$ is homogeneous. Passing to a Noether normalization of $R$, we reduce to the case where $R$ is a polynomial ring over $k$. We prove bounds for the mixed multiplicity of ideals in a polynomial ring, which through the reduction steps proves 1.2. For 1.1, by passing to the initial module, we may assume $E$ is a direct sum of ideals. Finally we use a formula relating the Buchsbaum-Rim multiplicity of the direct sum of ideals and the mixed multiplicity of those ideals \cite[{Theorem } 4.9]{br-mixed-mult} to obtain the desired bound. \\

\section{Preliminaries}\label{sec-prelim}
In this section we will define the Buchsbaum-Rim multiplicity of a module and the mixed multiplicity of ideals in a Noetherian local ring $(R,m,k)$. We will also define joint reductions and complete reductions as they play a key role in this paper. Throughout this section $d =$ dim $R$.

\begin{Def} \label{Def-BR-Mult}
Let $E \subseteq F = R^r$ be a submodule of a finite-rank free module with $\lambda(F/E) < \infty$. Consider the Rees algebra of $E$, denoted $\mathcal{R}[E]$, which is the image of the natural map $\textrm{Sym}(E) \to \textrm{Sym}(F)$ from the symmetric algebra of $E$ to the symmetric algebra of $F$. Define $E^n = \mathcal{R}[E]_n$ to be the $n$-th degree component of $\mathcal{R}[E]$. Similarly define $F^n = \textrm{Sym}_n(F)$ to be the $n$-th degree component of the symmetric algebra of $F$. Then one can consider $\lambda(F^n/E^n)$; if $E \subsetneq F$ this is a polynomial, $P(n)$, of degree $d+r-1$ when $n >> 0$. The \textit{Buchsbaum-Rim multiplicity} of $E$, denoted $br(E)$, is defined to be the normalized leading coefficient of $P(n)$, which is the following limit: $$br(E) = \lim \limits_{n \to \infty} (d+r-1)!\frac{\lambda(F^n/E^n)}{n^{d+r-1}}.$$ We will sometimes write $br(E)$ as $br_R(E)$ to emphasize that we are regarding $E$ as an $R$-module when taking lengths in the above definition. 
\end{Def}

Notice that if $r = 1$ and $E \subsetneq F$ then $E$ is an $m$-primary ideal of $R$, $F^n \cong R$, and $E^n$ is isomorphic to the $n$-th power of $E$ regarded as an ideal in $R$. In this case, $br(E) = e(E)$; hence the Buchsbaum-Rim multiplicity generalizes the Hilbert-Samuel multiplicity. 

The Hilbert-Samuel multiplicity is defined for $m$-primary ideals $I$ on a finite $R$-module $N$. This can also be done for the Buchsbaum-Rim multiplicity. 

\begin{Def}\label{Def-Gen-BR-Mult} Let $N$ be a finite $R$-module. With the notation as in \ref{Def-BR-Mult}, the \textit{Buchsbaum-Rim multiplicity of $E$ on $N$}, denoted as $br(E,N)$, is defined as the following limit $$br(E,N) = \lim \limits_{n \to \infty} (d+r-1)!\frac{\lambda\big(\frac{F^n}{E^n} \otimes N \big)}{n^{d+r-1}}.$$

\end{Def} We are mainly concerned with the version of the Buchsbaum-Rim multiplicity supplied in \ref{Def-BR-Mult}. However we give the more general version in \ref{Def-Gen-BR-Mult} since we will consider it in a technical step in the next section.

The associativity formula for the Buchsbaum-Rim multiplicity as defined in \ref{Def-BR-Mult} is known and has been proven by Kleiman \cite[{Proposition } 7]{KleimanSteven-2017}. We believe an analogous formula for the Buchsbaum-Rim multiplicity in definition \ref{Def-Gen-BR-Mult} is also known. However the only source that we are able to find for this result is in Validashti's thesis \cite[{Theorem } 6.5.1]{JV-Thesis}. In his thesis he proves an associativity formula for the $j$-multiplicity, which generalizes the Buchsbaum-Rim multiplicity. As we are not concerned with the $j$-multiplicity in this paper we will provide our own proof of the associativity formula.

\begin{Th}\label{BR-associatvity} Let $E \subseteq F = R^r$ be a submodule of a finite-rank free module with $\lambda(F/E) < \infty$. Let $N$ be a finite $R$-module. Let $\Lambda = \{ p \in Supp(N) \; : \; \textrm{dim } R/p = d\}$. Then $$br(E,N) = \sum_{p \in \Lambda} \lambda_{R_p}(N_p)br(E, R/p).$$

\begin{proof}
If $E = F$, then both sides of the equality is zero, hence we may assume $E \subsetneq F$. Consider a prime filtration of $N$, $0 = N_0 \subsetneq ... \subsetneq N_k = N$, where $N_{i}/N_{i-1} \cong R/p_i$ for some prime $p_i \in \textrm{Supp}(N)$. Then for each $1\leq i \leq k$ we have a short exact sequence $0 \to N_{i-1} \to N_i \to R/p_i \to 0$. By \cite[{Proposition} 3.8 (3)]{Buchsbaum-Rim}, $br(E, N_i) = br(E, N_{i-1}) + br(E, R/p_i)$. Hence from the filtration we get $$br(E,N) = \sum_{i=1}^k br(E, R/p_i).$$ 

By \cite[{Theorem} 3.4]{Buchsbaum-Rim} for a finite module $N$ and for $n >> 0$ the degree of the polynomial $\lambda(\frac{F^n}{E^n} \otimes N)$ is dim $N + r - 1$. Hence by definition $br(E, R/p_i) = 0$ if dim $R/p_i < d$. For each $p \in \Lambda$, by localizing the filtration at $p$ it can be seen that the number of times $p$ appears in the filtration is equal to $\lambda_{R_p}(N_p)$. Hence the above summation becomes $$\sum_{i=1}^k br(E, R/p_i) = \sum_{p \in \Lambda} \lambda_{R_p}(N_p)br(E, R/p).$$\end{proof}

\end{Th}

We now state the definitions of mixed multiplicity, joint reduction, and complete reduction.

\begin{Def}\cite[{Definition }17.4.3]{SwansonHuneke}\label{mixed-mult-def} Let $I_1,...,I_r$ be $m$-primary ideals, for $n_1,...,n_r>>0$, $\lambda(R/(I_1^{n_1}\cdots I_r^{n_r}))$ is a polynomial of degree $d$ in $r$ variables, $x_1,..., x_r$. For non-negative integers $a_1,...,a_r$ with $a_1+...+a_r=d$, the coefficient of the term $x_1^{a_1}\cdots x_r^{a_r}$ is $$\frac{1}{a_1!\cdots a_r!}e(I_1^{[a_1]},...,I_r^{[a_r]}).$$
We call $e(I_1^{[a_1]},...,I_r^{[a_r]})$ the mixed multiplicity of $I_1,...,I_r$ of type $(a_1,...,a_r)$. With each $I_i$ listed $a_i$ times, $e(I_1,...,I_1,...,I_r,...,I_r)$ also denotes the mixed multiplicity of $I_1,...,I_r$ of type $(a_1,...,a_r)$. The mixed multiplicity can be defined for a finite module $M$ with respect to $I_1,...,I_r$. In this case, the normalized coefficients are from the polynomial $\lambda(R/(I_1^{n_1}\cdots I_r^{n_r}) \otimes M)$, and the mixed multiplicity is denoted $e(I_1^{[a_1]},...,I_r^{[a_r]}; M)$.

\end{Def}

\begin{Def} \cite[{Definition }17.1.3]{SwansonHuneke} Let $I_1,...,I_r$ be ideals of $R$. If $x_i \in I_i$ and \\$\sum_{i=1}^r x_i I_1\cdots \hat{I_i} \cdots I_r$ is a reduction of $I_1\cdots I_r$, then $x_1,...,x_r$ is a $\textit{joint reduction}$ of $I_1,...,I_r$. 
\end{Def}

For an ideal $I$, let $\overline{I}$ denote the integral closure of $I$. As with the Hilbert-Samuel multiplicity, the mixed multiplicity is invariant under integral closure. This can be seen through joint reduction. Since $I_1\cdots I_r$ is a reduction of $\overline{I_1}\cdots\overline{I_r}$, it follows that if $x_1,...,x_r$ is a joint reduction of $I_1,...,I_r$ then it is also a joint reduction of $\overline{I_1},...,\overline{I_r}$. It follows from \cite[{Theorem }2.4]{rees84} that $e(I_1^{[a_1]},...,I_r^{[a_r]}) = e(\overline{I_1}^{[a_1]},...,\overline{I_r}^{[a_r]})$. This allows us to replace ideals with their integral closures when dealing with mixed multiplicities. 

\begin{Def} \cite[{p. }402]{rees84} Let $U = (I_1,...,I_r)$ be a set of, not necessarily distinct, ideals of $R$. The set of elements $\{x_{ij} : 1 \leq i \leq r, 1 \leq j \leq d, x_{ij} \in I_i \}$, with $y_j = x_{1j}\cdots x_{rj}$, is a $\textit{complete reduction}$ of $U$ if $(y_1,...,y_d)$ is a reduction of $I_1\cdots I_r$.

\end{Def}

Rees showed that complete reductions exist if $k$ is infinite \cite[{Theorem }1.3]{rees84}. Furthermore complete reductions are related to joint reductions by the following corollary.

\begin{Cor}\label{general-implies-joint} \cite[{Corollary }(i)]{rees84} Let $r \geq d$ and $I_1,...,I_r$ be, not necessarily distinct, ideals of $R$ such that $I_1\cdots I_r$ is not contained in any minimal prime ideal of $R$. Let $I_{n_1},...,I_{n_d}$ be a subset of $I_1,...,I_r$. Let $\{x_{ij}\}$ be a complete reduction of $(I_1,...,I_r)$. Set $x_j = x_{n_j j}$, then $x_1,...,x_d$ is a joint reduction of $I_{n_1},...,I_{n_d}.$
\end{Cor}

The next theorem shows how joint reductions can be used to reduce the dimension of the ring.

\begin{Th}\label{mod-gen-elem} \cite{teissier1973cycles} \textnormal{(as cited in \cite[{Theorem }17.4.6]{SwansonHuneke})} Suppose $d \geq 2$ and $|k| = \infty$. Let $I_1,...,I_d$ be $m$-primary ideals of $R$ and let $x_1,...,x_d$ be a joint reduction of $I_1,...,I_d$ where $x_1$ is a general element. Denote by $-'$ images in $R/(x_1)$, then
$$e(I_1,...,I_d) = e(I_2',...,I_d').$$
\end{Th}

We include the next lemma of Rees as it plays a crucial role in our calculations. We will use it often, in tandem with \ref{mod-gen-elem}, in section \ref{sec-polyring}.

\begin{Lemma}\label{split-mixed-mult} \cite[{Lemma }2.5]{rees84} Let $I_1,...,I_d,J$ be $m$-primary ideals of $R$, then $$e(I_1J,I_2,...,I_d) = e(I_1,I_2,...,I_d) + e(J,I_2,...,I_d).$$

\end{Lemma}



$\newline$
\section{Reduction Steps} \label{sec-reduction-steps}

The reduction steps are essentially due to Lech \cite{LechIneq}, but we will follow the more modern approach in \cite[{Theorem }3.1]{HSV-Lech}. Let $(R,m,k)$ be a Noetherian local ring of dimension $d$, $E \subseteq F=R^r$ a submodule of a finite-rank free module with $\lambda (F/E) < \infty$. 
Now let $G$ be the associated graded ring of $R$, gr$_m(R)$. For a module $M$, let $gr_m(M)$ be the associated graded module of $M$. Define $$E^*=\bigoplus_{i \geq 0} \frac{E \cap m^iF+m^{i+1}F}{m^{i+1}F} \subseteq \textrm{gr}_m(F)$$
which is a homogeneous submodule of gr$_m(F)$. Similarly, define $$(E^n)^*=\bigoplus_{i \geq 0} \frac{E^n\cap m^iF^n+m^{i+1}F^n}{m^{i+1}F^n} \subseteq \text{gr}_m(\textrm{Sym}_n(F))=\textrm{Sym}_n(\text{gr}_m(F)).$$
Let $G_{+}$ denote the maximal homogeneous ideal of $G$. The next result compares the Buchsbaum-Rim multiplicity of $E \subseteq F$ with that of $E^* \subseteq G^r$. Furthermore if the bound in 1.1 holds for $G_+E^*$ then it also holds for $mE$.

\begin{Prop}\label{br-reduction-to-graded}

Let $(R,m)$ be a Noetherian local ring with dim $R = d$, and $E \subseteq F = R^r$ a submodule of a finite-rank free module with $\lambda(F/E) < \infty$, then 

$$br_R(mE) \leq br_G(G_+E^*).$$

Furthermore, if the bound $$br_G(G_+E^*) < \frac{(d+r-1)!}{r!}\lambda(\text{gr}_m(F)/E^*)e(G)$$
holds, then the following bound also holds, $$br_R(mE) < \frac{(d+r-1)!}{r!}\lambda(F/E)e(R).$$
\end{Prop}

\begin{proof} Comparing elements, we see that $(E^*)^n \subseteq (E^n)^*$, and $G_+E^* \subseteq (mE)^*$. To show $$br_R(mE) \leq  br_G(G_+E^*)$$ we will show that for any $n$ $$\lambda_R\left(\frac{F^n}{(mE)^n}\right) \leq \lambda_G\left(\frac{\textrm{Sym}_n(\text{gr}_m(F))}{(G_+E^*)^n}\right).$$
First we have $$\lambda_R\left(\frac{F^n}{(mE)^n}\right) = \lambda_G\left(\text{gr}_m\left(\frac{F^n}{(mE)^n}\right)\right) = \lambda_G\left(\frac{\textrm{Sym}_n(\text{gr}_m(F))}{((mE)^n)^*} \right). $$
Now since $((mE)^*)^n \subseteq ((mE)^n)^*$ and $(G_+E^*)^n \subseteq ((mE)^*)^n $ we have 

$$\lambda_G\left(\frac{\textrm{Sym}_n(\text{gr}_m(F))}{((mE)^n)^*} \right) \leq \lambda_G\left(\frac{\textrm{Sym}_n(\text{gr}_m(F))}{((mE)^{*})^n}\right)\leq \lambda_G\left(\frac{\textrm{Sym}_n(\text{gr}_m(F))}{(G_+E^*)^n}\right).$$

Next we prove the second statement. Notice that $\lambda_G(\frac{\text{gr}_m(F)}{E^*})=\lambda_G(\text{gr}_m(F/E))= \lambda_R(F/E)$ and $e(R)=e(G)$, hence the constants on the right hand side of the two inequalities are the same. Now from the first part

$$br_R(mE) \leq br_G(G_+E^*) < \frac{(d+r-1)!}{r!}\lambda(\text{gr}_m(F)/E^*)e(G) = \frac{(d+r-1)!}{r!}\lambda(F/E)e(R).$$ \end{proof}
For mixed multiplicities, we can define $I^*$ by replacing $E$ with $I$ and $F$ with $R$. We now prove a result analogous to \ref{br-reduction-to-graded} for mixed multiplicities.

\begin{Prop} \label{mixed-reduction-to-graded} Let $(R,m,k)$ be a Noetherian local ring with infinite residue field $k$ and dim $R = d$. Let $I_1,...,I_d$ $m$-primary ideals, then $$e_R(mI_1,...,mI_d) \leq e_G(G_+I_1^*,...,G_+I_d^*).$$

Furthermore, if the bound $$e_G(G_+I_1^*,...,G_+I_d^*) < (d-1)! \sum_{i=1}^{d}\lambda(G/I_i^*)e(G) $$ holds, then the following bound also holds, $$e_R(mI_1,...,mI_d) < (d-1)! \sum_{i=1}^{d}\lambda(R/I_i)e(R).$$

\end{Prop} 
\begin{proof}
Let $f_1^*,...,f_d^*$ be a joint reduction of $I_1^*,...,I_d^*$, then by \cite[{Theorem }2.4]{rees84} (as cited in \cite[{Theorem }17.4.9]{SwansonHuneke}), $e(f_1^*,...,f_d^*) = e(I_1^*,...,I_d^*)$. Now let $f_1,...,f_d$ with each $f_i \in I_i$ be a lift of $f_1^*,...,f_d^*$. Notice that $(f_1,...,f_d)$ is an $m$-primary ideal because $\lambda_R(R/(f_1,...,f_d)) = \lambda_G(G/(f_1,...,f_d)^*) \leq \lambda_G(G/(f_1^*,...,f_d^*)) < \infty$. Then from \cite[{Lemma }2.8]{swanson} 

$$e_R(I_1,...,I_d) \leq e_R(f_1,...,f_d) \leq e_G((f_1,...,f_d)^*) \leq e_G(f_1^*,...,f_d^*) = e_G(I_1^*,...,I_d^*).$$

As the above inequality holds for any $m$-primary ideals, we apply it to $mI_1,...,mI_d$ to get $$e_R(mI_1,...,mI_d) \leq e_G((mI_1)^*, ..., (mI_d)^*) \leq e_G(m^*I_1^*,...,m^*I_d^*) = e_G(G_+I_1^*,...,G_+I_d^*).$$

The second part follows in exactly the same way as the second part of \ref{br-reduction-to-graded}. We have $\lambda(G/I_i^*) = \lambda(R/I_i)$ and $e(G) = e(R)$. The inequality from the first part shows that

\begin{align*}
    e_R(mI_1,...,mI_d) \leq e_G(G_+I_1^*,...,G_+I_d^*) &< (d-1)! \sum_{i=1}^{d}\lambda(G/I_i^*)e(G) \\ &= (d-1)! \sum_{i=1}^{d}\lambda(R/I_i)e(R). 
\end{align*}

\end{proof}

To prove 1.1 and 1.2, the above two results allow us to replace $R$ with $G$, $E \subseteq F$ with $E^* \subseteq G^r$, and $I_i$ with $I_i^*$ to assume $R$ is a standard graded ring over an infinite field, $E$ is a graded submodule, and each $I_i$ is a homogeneous ideal. From here, we reduce to the case where $R$ is a polynomial ring over an infinite field.

\begin{Prop} \label{reduction-to-poly} Let $R$ be a standard graded ring over an infinite field $k$ and let $m$ be the maximal homogeneous ideal of $R$. Let $S=k[x_1,...,x_d]$ be a homogeneous Noether normalization of $R$, such that with $n = (x_1,...,x_d)S$, $nR$ is a reduction of $m$. Let $E \subseteq F = R^r$ be a submodule of a finite-rank free module with $\lambda(F/E) < \infty$. Let $\mathcal{F} = S^r \subseteq F$ and set $\mathcal{E} = E \cap \mathcal{F}$. Then the inequality $$br_S(n\mathcal{E}) < \frac{(d+r-1)!}{r!}\lambda(\mathcal{F}/\mathcal{E}) $$ implies the inequality $$br_R(mE) < \frac{(d+r-1)!}{r!}\lambda(F/E)e(R).$$

Now let $I_1,...,I_d$ be $m$-primary ideals of $R$ and set $J_i = I_i \cap S$, then the inequality $$e_S(nJ_1,...,nJ_d) < (d-1)!\sum_{i=1}^{d} \lambda(S/J_i)$$ implies the inequality $$e_R(mI_1,...,mI_d) < (d-1)!\sum_{i=1}^{d} \lambda(R/I_i)e(R).$$

\end{Prop}
\begin{proof}

We first prove the implication of inequalities for the Buchsbaum-Rim multiplicities. Since $\frac{\mathcal{F}}{\mathcal{E}} \hookrightarrow{} \frac{F}{E}$ we have $\lambda_S(\mathcal{F}/\mathcal{E}) \leq \lambda_S(F/E) = \lambda_R(F/E).$ 

We view $R$ as an $S$-module to show that $br_S(n\mathcal{E}, R) = br_R(n\mathcal{E}R)$. Notice dim $R$ = dim $S$ and rank$_R(F)$ = rank$_S(\mathcal{F})$. Hence it is enough to show that for any $i$ $$\lambda_S\left(\frac{\mathcal{F}^i}{(n\mathcal{E})^i} \tens{S} R \right) = \lambda_R\left(\frac{F^i}{(n\mathcal{E}R)^i }\right).$$ But this follows as the two modules are isomorphic as $R$ modules and lengths of modules over $S$ and $R$ coincide.

Next, since $\mathcal{E} R \subseteq E$ and $nR \subseteq m$, we have \begin{equation*}
    br_R(mE) \leq br_R(n\mathcal{E} R) = br_S(n\mathcal{E}, R).
\end{equation*} We now apply the associativity formula in \ref{BR-associatvity}. Since $S$ is a domain, the formula gives 
$$br_S(n\mathcal{E}, R) = br_S(n\mathcal{E})\text{rank}_S(R).$$
Applying the associativity formula for Hilbert-Samuel multiplicity we have $\text{rank}_S(R) = e_S(n, R) = e(R)$. Combining this with the above line shows $br_S(n\mathcal{E}, R) = br_S(n\mathcal{E})e(R)$. Hence, $br_R(mE) \leq br_S(n\mathcal{E})e(R).$ 

Now if the bound holds for $\mathcal{E} \subseteq \mathcal{F}$ then, because $\lambda_S(\mathcal{F}/\mathcal{E}) \leq \lambda_R(F/E)$, we have
$$br_R(mE) \leq br_S(n\mathcal{E})e(R) < \frac{(d+r-1)!}{r!}\lambda_S(\mathcal{F}/\mathcal{E})e(R) \leq \frac{(d+r-1)!}{r!}\lambda_R(F/E)e(R).$$ 

We now provide a similar argument for the statement about mixed multiplicity. Let $J_i = I_i \cap S$, first we have $e_R(mI_1,...,mI_d) \leq e_R(nJ_1R,...,nJ_dR) = e_S(nJ_1,...,nJ_d; R)$. Using the associativity formula for mixed multiplicity \cite[{Theorem }17.4.8]{SwansonHuneke} we have 

$$e_S(nJ_1R,...,nJ_d;R) = e_S(nJ_1,...,nJ_d)\text{rank}_S(R) = e_S(nJ_1,...,nJ_d)e(R).$$

As for the colengths of the ideals, since $S/J_i \hookrightarrow R/I_i$ we have $\lambda(S/J_i) \leq \lambda(R/I_i)$. Now suppose the bound holds for $J_1,...,J_d$, then
\begin{align*} e_R(mI_1,...,mI_d) &\leq e_S(nJ_1,...,nJ_d)e(R) \\ &< (d-1)!\sum_{i=1}^{d} \lambda(S/J_i)e(R) \leq (d-1)!\sum_{i=1}^{d} \lambda(R/I_i)e(R).
\end{align*}
Hence the bound also holds for $I_1,...,I_d$. \end{proof}

We now put together the above results to show that our bounds in the case when $R$ is a Noetherian local ring reduce to the case when $R$ is a polynomial ring.

\begin{Th} \label{reduction-step}

Assume that for any polynomial ring $S = k[x_1,...,x_d]$ over an infinite field $k$, with $n = (x_1,...,x_d)$, and any $\mathcal{E} \subseteq \mathcal{F} = S^r$ a submodule of a finite-rank free $S$-module with $\lambda(\mathcal{F}/\mathcal{E}) < \infty$ that the following inequality holds

$$br_S(n\mathcal{E}) < \frac{(d+r-1)!}{r!}\lambda(\mathcal{F}/\mathcal{E}).$$

Then the following inequality holds for any Noetherian local ring $(R,m)$ with dim $R = d$ and any $E \subseteq F = R^r$ a submodule of a finite-rank free module with $\lambda(F/E) < \infty,$

$$br_R(mE) < \frac{(d+r-1)!}{r!}\lambda(F/E)e(R).$$

Now assume that the following inequality holds for any $n$-primary $S$-ideals $J_1,...,J_d$ $$e_S(nJ_1,...,nJ_d) < (d-1)\sum_{i=1}^{d} \lambda(S/J_i).$$ Then the following inequality holds for any $m$-primary $R$-ideals $I_1,...,I_d$

$$e_R(mI_1,...,mI_d) < (d-1)\sum_{i=1}^{d} \lambda(R/I_i).$$

\end{Th}

\begin{proof}
We use a standard technique to assume the residue field $R/m = k$ is infinite. Let $R(X)=R[X]_{mR[X]}$, then $R \rightarrow R(X)$ is a faithfully flat extension of local rings with the same dimensions and multiplicities. Moreover, for any $R$-module $M$, $\lambda_R(M)=\lambda_{R(X)}(M \otimes R(X))$. Hence, replacing $R$ with $R(X)$, we may assume that $k$ is infinite. 
 
Let $G = gr_m(R)$ and $G_+$ be the maximal homogeneous ideal of $G$, then $G$ is a standard graded ring over $k$. Let $S$ be a Noether normalization of $G$. Then $S$ is isomorphic to a polynomial ring over $k$ of dimension $d$. Applying \ref{reduction-to-poly}, the bounds for the Buchsbaum-Rim and mixed multiplicities hold for $G$. Finally from \ref{br-reduction-to-graded} and \ref{mixed-reduction-to-graded}, since the bounds for the multiplicities hold for $G$ they also hold for $R$.
\end{proof}

In the next section we will prove Lech type bounds for the mixed multiplicity of ideals in a polynomial ring. To obtain our results in the polynomial ring case we employ extensively Rees's notion of complete reductions \cite{rees84} defined in section \ref{sec-prelim}.
$\newline$

\section{Mixed Multiplicity Bounds in Polynomial Rings}\label{sec-polyring}

In this section we prove a Lech type bound for the mixed multiplicities of ideals in a polynomial ring. We first prove technical bounds for dimensions 2 and 3. For higher dimensions we use complete reductions to reduce to the low dimensional case. Throughout this section $k$ is assumed to be infinite and $m$ denotes the homogeneous maximal ideal.

\begin{Prop} \label{dim2-ineq}

Let $R = k[x_1,x_2]$ and $I_1,...,I_r$ be $m$-primary ideals with $r \geq 2$. Let $x$ be a general linear form and denote $-'$ to be images in $R' = R/(x)$. Then 
$$2\sum_{1\leq i < j \leq r} e(I_i, I_j) + (r-1) \sum_{i=1}^r e(I_i') \leq 2(r-1) \sum_{i=1}^r \lambda(R/I_i).$$

\end{Prop}

\begin{proof}
First, because $2e(I_i,I_j) \leq e(I_i)+e(I_j)$ \cite[p. 365]{SwansonHuneke}, we have
$$2\sum_{1 \leq i < j \leq r} e(I_i,I_j) \leq (r-1) \sum_{i=1}^r e(I_i).$$
Using this inequality, we have reduced to showing $$(r-1)\sum_{i=1}^r (e(I_i)+e(I_i')) \leq 2(r-1)\sum_{i=1}^r \lambda(R/I_i).$$
From \cite[{Corollary }4.5]{HSV-Lech}, we have $e(I_i)+e(I_i') \leq 2 \lambda(R/I_i)$, which gives the result. \end{proof}

When all of the ideals are powers of the maximal ideal the two bounds in the proof are sharp, hence our bound is also sharp in this case.

Next, \ref{lech-mixed} is a Lech bound for mixed multiplicity valid in any dimension. We will need this result to deal with a particular case in the proof of \ref{dim3-ineq}.

\begin{Prop} \label{lech-mixed}
Let $R= k[x_1,...,x_d]$ and $I_1,...,I_d$ be $m$-primary ideals. Then 
$$e(I_1,...,I_d) \leq (d-1)!\sum_{i=1}^{d} \lambda(R/I_i).$$
\end{Prop}

\begin{proof}
We may assume $I_1,...,I_d$ are integrally closed. We induct on $d$. The base case where $d=1$ is clear. Let $d > 1$, and we will induct on $\sum_{i=1}^{d} \lambda(R/I_i)$. The base case is when each $I_i = m$. This holds since $e(m,...,m) = 1$. For the induction, let $x \in m $ be a general linear form not contained in any $I_i \neq m$. Let $-'$ denote images in $R' = R/(x)$.  We may choose $x$ to be $x_{11}$ in a complete reduction of $m, I_1,...,I_d$. Without loss of generality, we may assume that $\lambda(R'/I_1') = \textrm{max}\{\lambda(R'/I_i') : 1 \leq i \leq d \}$. Now by \ref{general-implies-joint}, $x$ is part of a joint reduction of $m, I_2, ..., I_d$.

Set $\tilde{I_1} = I_1 : x$. By \cite[{Theorem }2.3]{HSV-Lech} $m\tilde{I_1} \subseteq I_1$, hence $e(I_1,...,I_d) \leq e(m\tilde{I_1},...,I_d)$. Expanding the mixed multiplicity using \ref{split-mixed-mult} and applying \ref{mod-gen-elem}, we have
$$e(I_1,...,I_d) \leq e(m\tilde{I_1},...,I_d) = e(m,I_2,...,I_d) + e(\tilde{I_1},...,I_d) = e(I_2',...,I_d') + e(\tilde{I_1},...,I_d). $$
By induction on $d$, 
$$e(I_2',...,I_d')\leq (d-2)! \sum_{i=2}^{d} \lambda(R'/I_i').$$
By induction on $\sum_{i=1}^d \lambda(R/I_i)$, $$e(\tilde{I_1},...,I_d) \leq (d-1)!\Big( \lambda(R/\tilde{I_1}) + \sum_{i=2}^{d} \lambda(R/I_i) \Big).$$
Using these two inequalities from induction, we have
$$ e(I_1,...,I_d) \leq (d-2)! \sum_{i=2}^{d} \lambda(R'/I_i') + (d-1)!\Big( \lambda(R/\tilde{I_1}) + \sum_{i=2}^{d} \lambda(R/I_i) \Big). $$
Now $\lambda(R'/I_1') = \textrm{max}\{\lambda(R'/I_i') : 1 \leq i \leq d \}$ gives
$$(d-2)! \sum_{i=2}^{d} \lambda(R'/I_i') \leq  (d-1)!\lambda(R'/I_1').$$
Hence
$$ e(I_1,...,I_d) \leq (d-1)!\lambda(R'/I_1') + (d-1)!\Big( \lambda(R/\tilde{I_1}) + \sum_{i=2}^{d} \lambda(R/I_i) \Big). $$
Finally by \cite[{Lemma }2.6]{HSV-Lech}, $\lambda(R'/I_1') + \lambda(R/\tilde{I_1}) = \lambda(R/I_1)$, which gives the result.
\end{proof}
Now we can use the above inequality and the dimension two result to prove a Lech type bound for a polynomial ring in three variables.

\begin{Prop} \label{dim3-ineq}
Let $R=k[x_1,x_2,x_3]$ and $I_1,...,I_4$  be $m$-primary ideals. Let $x,y$ be general linear forms. Let $-'$ and $-''$ denote images in $R' = R/(x)$ and $R'' = R/(x,y)$ respectively. Then $$\sum_{1 \leq i < j < k \leq 4} e(I_i, I_j, I_k) + \sum_{1 \leq i < j \leq 4} e(I_i', I_j') + \sum_{i=1}^4 e(I_i'')+1 \leq 6\sum_{i=1}^4 \lambda(R/I_i).$$

\end{Prop}

\begin{proof}

We may assume each $I_i$ is integrally closed. We induct on $\sum_{i=1}^4 \lambda(R/I_i)$. In the base case, each $I_i=m$, and the result holds. We now have two cases. In the first case, where none of the ideals are $m$, we will continue the induction. In the latter case, where at least one ideal is $m$, we will show the inequality directly. \\

(1) Suppose $I_i \neq m$ for all $i$. Let $\tilde{I_i}=I_i:x$. As each $I_i$ is integrally closed, by \cite[{Theorem }2.3]{HSV-Lech} $m\tilde{I_i} \subseteq I_i$. Hence the mixed multiplicity cannot decrease whenever we replace $I_i$ with $m\tilde{I_i}$. We will use this to estimate the terms in the summation $\sum e(I_i, I_j, I_k)$. 

First, we replace $I_1$, $I_2$ with $m\tilde{I_1}$, $m\tilde{I_2}$. Then we apply \ref{split-mixed-mult} and \ref{mod-gen-elem} to the mixed multiplicities. 
\begin{align*}
 \sum_{1 \leq i < j < k \leq 4} e(I_i, &I_j, I_k) 
    \leq e(m\tilde{I_1},m\tilde{I_2},I_3) +  e(m\tilde{I_1},m\tilde{I_2},I_4) +  e(m\tilde{I_1},I_3,I_4) + e(m\tilde{I_2},I_3,I_4) \\
    &\leq e(\tilde{I_1},\tilde{I_2},I_3) + e(\tilde{I_1},\tilde{I_2},I_4)+ e(\tilde{I_1},I_3,I_4)+ e(\tilde{I_2},I_3,I_4)  \\
    & \quad + e(\tilde{I_1}',I_3') + e(\tilde{I_1}',I_4') + e(\tilde{I_2}',I_3') + e(\tilde{I_2}',I_4') + 2e(I_3',I_4') +  e(I_3'')+e(I_4'').
\end{align*}
Now we replace $I_3$, $I_4$ with $m\tilde{I_3}$, $m\tilde{I_4}$ and apply the two results again to obtain the following. To simplify the calculations, we have not changed the last line and will rearrange terms later.
\begin{align*}
\sum_{1 \leq i < j < k \leq 4} e(I_i, &I_j, I_k) \leq e(\tilde{I_1},\tilde{I_2},m\tilde{I_3}) + e(\tilde{I_1},\tilde{I_2},m\tilde{I_4})+ e(\tilde{I_1},m\tilde{I_3},m\tilde{I_4})+ e(\tilde{I_2},m\tilde{I_3},m\tilde{I_4}) \\
    & \quad + e(\tilde{I_1}',I_3') + e(\tilde{I_1}',I_4') + e(\tilde{I_2}',I_3') + e(\tilde{I_2}',I_4') + 2e(I_3',I_4') +  e(I_3'')+e(I_4'') \\
    &\leq e(\tilde{I_1},\tilde{I_2},\tilde{I_3}) + e(\tilde{I_1},\tilde{I_2},\tilde{I_4}) + e(\tilde{I_1},\tilde{I_3},\tilde{I_4}) + e(\tilde{I_2},\tilde{I_3},\tilde{I_4}) + 2e(\tilde{I_1}',\tilde{I_2}')\\ 
    &\quad + e(\tilde{I_1}',\tilde{I_3}') + e(\tilde{I_1}',\tilde{I_4}') + e(\tilde{I_2}',\tilde{I_3}') + e(\tilde{I_2}',\tilde{I_4}') + e(\tilde{I_1}'') + e(\tilde{I_2}'') \\
    &\quad + e(\tilde{I_1}',I_3') + e(\tilde{I_1}',I_4') + e(\tilde{I_2}',I_3') + e(\tilde{I_2}',I_4') + 2e(I_3',I_4') +  e(I_3'')+e(I_4'').
\end{align*}
Rearranging terms based on whether the ideals are in $R$, $R'$, or $R''$, we have
\begin{align*}
    \sum_{1 \leq i < j < k \leq 4} e(I_i, &I_j, I_k) \leq \sum_{1 \leq i < j < k \leq 4} e(\tilde{I_i},\tilde{I_j},\tilde{I_k}) \\ &\quad + 2e(\tilde{I_1}',\tilde{I_2}') + e(\tilde{I_1}',\tilde{I_3}') + e(\tilde{I_1}',\tilde{I_4}') + e(\tilde{I_2}',\tilde{I_3}')  + e(\tilde{I_2}',\tilde{I_4}') \\ 
    &\quad   + e(\tilde{I_1}',I_3') + e(\tilde{I_1}',I_4') + e(\tilde{I_2}',I_3') + e(\tilde{I_2}',I_4') + 2e(I_3',I_4') \\
    &\quad  + e(\tilde{I_1}'') + e(\tilde{I_2}'') +  e(I_3'')+e(I_4'').
\end{align*}
We will be able to use induction on $\sum_{i=1}^4 \lambda(R/I_i)$ after estimating the term $2e(I_3',I_4')$ using the same technique of applying \ref{split-mixed-mult} and \ref{mod-gen-elem}.
\begin{align*}
    2e(I_3',I_4') &= e(I_3', I_4')+e(I_3',I_4') \\
    &\leq e(m'\tilde{I_3}', I_4') + e(I_3', m'\tilde{I_4}') \\ 
    &\leq e(\tilde{I_3}',I_4') + e(I_4'') + e(I_3',\tilde{I_4}') + e(I_3'')  
\end{align*}    
Once more replacing $I_3'$, $I_4'$ with $m'\tilde{I_3}'$, $m'\tilde{I_4}'$, we have
\begin{align*}    
   2e(I_3',I_4') &\leq e(\tilde{I_3}',m'\tilde{I_4}') + e(I_4'') + e(m'\tilde{I_3}', \tilde{I_4}') + e(I_3'') \\
    &\leq 2e(\tilde{I_3}',\tilde{I_4}') + e(\tilde{I_3}'') + e(\tilde{I_4}'') + e(I_3'') + e(I_4'').
\end{align*}
Using this inequality and again grouping terms based on whether the ideals are in $R$, $R'$, or $R''$, we have
\begin{align*}
    \sum_{1 \leq i < j < k \leq 4} e(I_i, &I_j, I_k) \leq \sum_{1 \leq i < j < k \leq 4} e(\tilde{I_i},\tilde{I_j},\tilde{I_k}) + \sum_{1 \leq i < j \leq 4} e(\tilde{I_i}', \tilde{I_j}') \\ &\quad + e(\tilde{I_1}', \tilde{I_2}')+ e(\tilde{I_3}', \tilde{I_4}') + e(\tilde{I_1}',I_3') + e(\tilde{I_1}',I_4') + e(\tilde{I_2}',I_3') + e(\tilde{I_2}',I_4') \\
    &\quad + \sum_{i=1}^4 e(\tilde{I_i}'') + 2e(I_3'')+2e(I_4'').
\end{align*}
Now, by induction on $\sum_{i=1}^4 \lambda(R/I_i)$,
\begin{align*}
    \sum_{1 \leq i < j < k \leq 4} e(I_i, &I_j, I_k) \leq  6\sum_{i=1}^4 \lambda(R/\tilde{I_i})-1+ e(\tilde{I_1}', \tilde{I_2}')+ e(\tilde{I_3}', \tilde{I_4}') + e(\tilde{I_1}',I_3') + e(\tilde{I_1}',I_4') \\
    &\quad + e(\tilde{I_2}',I_3') + e(\tilde{I_2}',I_4') + 2e(I_3'')+2e(I_4'').
\end{align*}
Because $I_i \subseteq \tilde{I_i}$, we have $e(\tilde{I_i}',I_j')\leq e(I_i',I_j')$ and $e(\tilde{I_i}',\tilde{I_j}')\leq e(I_i',I_j')$. This gives
\begin{align*}
    \sum_{1 \leq i < j < k \leq 4} e(I_i, I_j, I_k) &\leq  6\sum_{i=1}^4 \lambda(R/\tilde{I_i})-1+ \sum_{1\leq i <j \leq 4} e(I_i', I_j')+ 2e(I_3'')+2e(I_4'') \\
    &\leq 6\sum_{i=1}^4 \lambda(R/\tilde{I_i})-1+ \sum_{1\leq i <j \leq 4} e(I_i', I_j')+ 2\sum_{i=1}^4 e(I_i'').
\end{align*}
Finally we have
\begin{align*}
\sum_{1 \leq i < j < k \leq 4} e(I_i, I_j, I_k)\; + &\sum_{1 \leq i < j \leq 4} e(I_i', I_j') + \sum_{i=1}^4 e(I_i'')+1\\
&\leq 6\sum_{i=1}^4 \lambda(R/\tilde{I_i}) + 2\sum_{1 \leq i < j \leq 4} e(I_i', I_j') + 3\sum_{i=1}^{4} e(I_i'').
\end{align*}
The result follows by first applying \ref{dim2-ineq} with $r=4$ to the last two summations, and then applying $\lambda(R'/I_i')+\lambda(R/\tilde{I_i})=\lambda(R/I_i)$ \cite[{Lemma }2.6]{HSV-Lech}.\\

(2) For the second case, assume that at least one $I_i$ is $m$. We choose $I_4 = m$. Then by \ref{mod-gen-elem}, we need to show

$$e(I_1,I_2,I_3) + 2\sum_{1\leq i < j \leq 3}e(I_i',I_j') + 2\sum_{i=1}^{3}e(I_i'') +2 \leq 6\sum_{i=1}^{3}\lambda(R/I_i) + 6.$$

Combining \cite[p. 365]{SwansonHuneke} and \cite[{Corollary }4.5]{HSV-Lech} gives $$2\sum_{1\leq i < j \leq 3}e(I_i',I_j') + 2\sum_{i=1}^{3}e(I_i'') \leq 4\sum_{i =1}^{3} \lambda(R/I_i).$$

So it is enough to show $e(I_1,I_2,I_3)+2 \leq 2\sum_{i=1}^{3}\lambda(R/I_i)+6$, which follows from \ref{lech-mixed}.
\end{proof}

The next result is a generalization of \cite[{Theorem }6.1]{HSV-Lech} for mixed multiplicities of ideals in a polynomial ring.
\begin{Th}\label{dim4-ineq}
Let $R=k[x_1,...,x_d]$ with $d \geq 4$. For $m$-primary ideals $I_1, ..., I_d$,  $$e(mI_1,...,mI_d) < (d-1)! \sum_{i=1}^d \lambda(R/I_i).$$
\end{Th}

\begin{proof}
We may assume that each $I_i$ is integrally closed. We proceed by induction on the dimension. The base case of $d=4$ will be handled last. Now assume $d>4$. We will induct on $\sum_{i=1}^d \lambda(R/I_i)$. The base case is where $\sum_{i=1}^d \lambda(R/I_i)=d$. Then every $I_i=m$, so $e(mI_1,...,mI_d)=e(m^2)=2^d < d!$ since $d > 4$.

For the induction, take $x \in m$ to be a general linear form not contained in any $I_i \neq m$. As in the beginning of the proof of \ref{lech-mixed}, we may assume that $I_1$ is not $m$ and that $\lambda(R'/I_1') = \textrm{max}\{\lambda(R'/I_i')\ : 1\leq i \leq d \}$. Define $\tilde{I_1}=I_1:x$.

Applying \ref{split-mixed-mult} and \ref{mod-gen-elem},
\begin{align*}
e(mI_1,...,mI_d) &=e(I_1,mI_2,...,mI_d)+e(m, mI_2,...,mI_d) \\
    &\leq e(m\tilde{I_1},mI_2,...,mI_d)+e(m'I_2',...,m'I_d').
\end{align*}
By induction on $\sum_{i=1}^d \lambda(R/I_i)$,
 $$e(m\tilde{I_1},mI_2,...,mI_d) < (d-1)!\sum_{i=2}^d \lambda(R/I_i)     +(d-1)!\lambda(R/\tilde{I_1}). $$
By induction on $d$ and $\lambda(R'/I_1') = \textrm{max}\{\lambda(R'/I_i')\ : 1\leq i \leq d \}$,
 \begin{align*}
     e(m'I_2',...,m'I_d') &< (d-2)!\sum_{i=2}^d \lambda(R'/I_i') \leq (d-1)! \lambda(R'/I_1').
 \end{align*}
Because $\lambda(R'/I_1')+\lambda(R/\tilde{I_1})=\lambda(R/I_1)$ \cite[{Lemma 2.6}]{HSV-Lech}, we have
\begin{align*}
     e(mI_1,...,mI_d) &< (d-1)!\sum_{i=2}^d \lambda(R/I_i)     +(d-1)!\lambda(R/\tilde{I_1})+ (d-1)! \lambda(R'/I_1') \\
     &= (d-1)! \sum_{i=1}^d \lambda(R/I_i).
\end{align*}

It remains to show the case for $d=4$. We will use complete reductions to reduce to the result in dimension 3. We induct on $\sum_{i=1}^4 \lambda(R/I_i)$. First, if $\sum_{i=1}^4 \lambda(R/I_i)=4$, then each $I_i=m$, so the result holds. Now, let $\{x_{ij}\}$ be a complete reduction of $(m,m,m,I_1,I_2,I_3,I_4)$ where $x_{11}, x_{22}$, and $x_{33}$ are general linear forms. Furthermore, we may choose $x_{11}$ not contained in any $I_i \neq m$. Let $-'$, $-''$, and $-'''$ denote images in $R' = R/(x_{11}), R'' = R/(x_{11},x_{22}),$ and $R''' = R/(x_{11},x_{22},x_{33})$, respectively. Again, we may assume that $I_1$ is not $m$ and that $I_1'$ has maximum colength in $R'$. Using \ref{split-mixed-mult} and \ref{mod-gen-elem} we expand the mixed multiplicity as 
\begin{align*} e(mI_1,mI_2,mI_3,mI_4)&=e(I_1,I_2,I_3,I_4) + \sum_{1 \leq i< j < k \leq 4 } e(I_i', I_j', I_k')  \\ & + \sum_{1 \leq i < j \leq 4} e(I_i'', I_j'') \;\;\;\;\;+ \sum_{i=1}^{4} e(I_i''')+1.
\end{align*} 

Set $\tilde{I_i}=I_i:x_{11}$. We have four cases to consider based on whether zero, one, two, or three $I_i$ are $m$. \\

(1) Suppose $I_i \neq m$ for all $i$. Then $\tilde{I_i} \neq R$ for all $i$. Using the induction hypothesis, we have
\begin{align*}
    e(mI_1,&mI_2,mI_3,mI_4)  \\
    &\leq e(m\tilde{I_1},m\tilde{I_2},m\tilde{I_3},m\tilde{I_4}) + \sum_{1 \leq i< j < k \leq 4} e(I_i', I_j', I_k') + \sum_{1 \leq i < j \leq 4} e(I_i'', I_j'') + \sum_{i=1}^{4} e(I_i''')+1 \\
    &< 6\sum_{i=1}^4 \lambda(R/\tilde{I_i}) + \sum_{1 \leq i< j < k \leq 4} e(I_i', I_j', I_k') + \sum_{1 \leq i < j \leq 4} e(I_i'', I_j'') + \sum_{i=1}^{4} e(I_i''')+1.
\end{align*}
We have the result, after applying \cite[{Lemma }2.6]{HSV-Lech}, if the following holds $$\sum_{1 \leq i< j < k \leq 4} e(I_i', I_j', I_k') + \sum_{1 \leq i < j \leq 4} e(I_i'', I_j'') + \sum_{i=1}^{4} e(I_i''')+1 \leq 6 \sum_{i=1}^4 \lambda(R'/I_i').$$ 
This is the result of \ref{dim3-ineq} and completes this case. \\

(2) Suppose $I_i=m$ for exactly one $i$. We may assume $I_4=m$, so then $\tilde{I_4}=R$. Then we have
\begin{align*} e(mI_1,&mI_2,mI_3,mI_4) \\ 
&\leq e(m\tilde{I_1},m\tilde{I_2},m\tilde{I_3},m) + \sum_{1 \leq i< j < k \leq 4} e(I_i', I_j', I_k') + \sum_{1 \leq i < j \leq 4} e(I_i'', I_j'') + \sum_{i=1}^{4} e(I_i''')+1.
\end{align*}
As in the first case, by using \cite[{Lemma }2.6]{HSV-Lech}, it suffices to show
$$e(m\tilde{I_1}, m\tilde{I_2},m\tilde{I_3}, m) < 6 \sum_{i=1}^4 \lambda(R/\tilde{I_i}).$$
Now, by \ref{split-mixed-mult} and \ref{mod-gen-elem}, we have 
$$e(m\tilde{I_1}, m\tilde{I_2},m\tilde{I_3}, m)= e(\tilde{I_1}', \tilde{I_2}', \tilde{I_3}')+ \sum_{1 \leq i < j \leq 3} e(\tilde{I_i}'', \tilde{I_j}'') + \sum_{i=1}^3 e(\tilde{I_i}''') +1.$$
Define $J_i'=\tilde{I_i}'$ for $i=1,2,3$. Define $J_4'=mR'$. Then, by \ref{dim3-ineq}, we have 
$$\sum_{1 \leq i<j<k \leq 4} e(J_i', J_j', J_k') + \sum_{1 \leq i<j \leq 4} e(J_i'', J_j'') + \sum_{i=1}^4 e(J_i''') + 1 \leq 6 \sum_{i=1}^4 \lambda(R'/J_i').$$
Now, because $J_4'=mR'$, we have the following two inequalities 
$$6\sum_{i=1}^4 \lambda(R'/J_i')\leq 6+6\sum_{i=1}^3 \lambda(R/\tilde{I_i}),$$
\begin{align*} 
e(\tilde{I_1}', \tilde{I_2}', \tilde{I_3}')+ \sum_{1 \leq i < j \leq 3}& e(\tilde{I_i}'', \tilde{I_j}'') + \sum_{i=1}^3 e(\tilde{I_i}''') +8  \\
&\leq \sum_{1 \leq i<j<k \leq 4} e(J_i', J_j', J_k') + \sum_{1 \leq i<j \leq 4} e(J_i'', J_j'') + \sum_{i=1}^4 e(J_i''') + 1.
\end{align*}
Hence, $$e(\tilde{I_1}', \tilde{I_2}', \tilde{I_3}')+ \sum_{1 \leq i < j \leq 3} e(\tilde{I_i}'', \tilde{I_j}'') + \sum_{i=1}^3 e(\tilde{I_i}''') +1<6 \sum_{i=1}^4 \lambda(R/\tilde{I_i}).$$

(3) Suppose $I_i=m$ for exactly two values of $i$. We may assume $I_3=I_4=m$. As previously, it suffices to show $$e(m\tilde{I_1}, m\tilde{I_2},m,m) < 6 \sum_{i=1}^4 \lambda(R/\tilde{I_i}).$$
By \ref{split-mixed-mult} and \ref{mod-gen-elem}, we have
$$e(m\tilde{I_1},m\tilde{I_2},m,m)=e(\tilde{I_1}'',\tilde{I_2}'')+ e(\tilde{I_1}''')+e(\tilde{I_2}''')+1.$$
As in the proof of \ref{dim2-ineq}, using \cite[p. 365]{SwansonHuneke} and \cite[{Corollary }4.5]{HSV-Lech}, we have 
$$e(\tilde{I_1}'',\tilde{I_2}'')+ e(\tilde{I_1}''')+e(\tilde{I_2}''')+1 \leq 2\sum_{i=1}^2 \lambda(R/\tilde{I_i}) +1.$$
Hence, 
$$e(m\tilde{I_1}, m\tilde{I_2},m,m) < 6 \sum_{i=1}^4 \lambda(R/\tilde{I_i}).$$

(4) Last, suppose $I_1 \subsetneq m$ and $I_i=m$ for $i=2,3,4$. Once more, it suffices to show $$e(m\tilde{I_1},m,m,m) < 6\sum_{i=1}^4 \lambda(R/\tilde{I_i}).$$
Now, $e(m\tilde{I_1},m,m,m)=1+e(\tilde{I_1},m,m,m)=1+e(\tilde{I_1}''')$ and $e(\tilde{I_1}''') \leq \lambda(R'''/\tilde{I_1}''')$. Hence, $e(m\tilde{I_1},m,m,m) \leq 2 \sum_{i=1}^4 \lambda(R/\tilde{I_i})$, which gives the desired result.
\end{proof} 
$\newline$
\section{Proof of the Main Theorems} \label{sec-main-theorems}

\begin{Th}  Let $(R,m)$ be a Noetherian local ring with dim $R = d \geq 4$, and $E \subseteq F = R^r$ a submodule of a finite-rank free module with $\lambda(F/E) < \infty$ and $E \subseteq mF$. Then $$br(mE) < \frac{(d+r-1)!}{r!}\lambda(F/E)e(R).$$
\end{Th}

\begin{proof} By \ref{reduction-step} we reduce to the case where $R$ is a polynomial ring over an infinite field. Now fix a monomial order on $F$. Under this ordering  consider the initial module of $E$, denoted $in(E)$. From a direct generalization of \cite[{Theorem }15.3]{eisenbud-book}, we have $\lambda(F/E)=\lambda(F/in(E))$ and $\lambda(F^n/E^n)=\lambda(F^n/in(E^n))$. Further, $\lambda(F^n/in(E^n)) \leq \lambda(F^n/in(E)^n)$ since $in(E)^n \subseteq in(E^n)$. Hence $br(E) \leq br(in(E))$ and we can replace $E$ with $in(E)$ to assume that $E = \oplus_{i=1}^{r}I_i$ is a direct sum of $m$-primary ideals. We now apply \cite[Theorem 4.9]{br-mixed-mult} to express $br(mE)$ as a sum of mixed multiplicities, and then bound the mixed multiplicities with \ref{dim4-ineq}. 

$$br(mE) = \sum_{\substack{a_1 + ... + a_r = d \\ a_1,...,a_r \geq 0}}e(mI_1^{[a_1]}, ..., mI_r^{[a_r]}) < (d-1)!\sum_{\substack{a_1 + ... + a_r = d \\ a_1,...,a_r \geq 0}} \bigg(\sum_{i=1}^{r}a_i\lambda(R/I_i) \bigg).$$

Writing the colengths as a vector $\lambda = \langle \lambda(R/I_1),...,\lambda(R/I_r) \rangle$, we rewrite the sum as 
$$\sum_{\substack{a_1 + ... + a_r = d \\ a_1,...,a_r \geq 0}} \bigg(\sum_{i=1}^{r}a_i\lambda(R/I_i) \bigg) = \sum_{\substack{a_1 + ... + a_r = d \\ a_1,...,a_r \geq 0}} \langle a_1,...,a_r \rangle \cdot \lambda = \lambda \cdot \sum _{\substack{a_1 + ... + a_r = d \\ a_1,...,a_r \geq 0}} \langle a_1,...,a_r \rangle.$$

Notice that each component of the sum of the vectors is the same. We will call this number $c$.
$$ \langle c,...,c \rangle=\sum _{\substack{a_1 + ... + a_r = d \\ a_1,...,a_r \geq 0}} \langle a_1,...,a_r \rangle.$$

To compute $c$, we sum the components of the vectors in the above equation.
$$rc = \sum _{\substack{a_1 + ... + a_r = d \\ a_1,...,a_r \geq 0}} a_1 + ... + a_r = \sum _{\substack{a_1 + ... + a_r = d \\ a_1,...,a_r \geq 0}} d\;.$$ 
The above sum equals $d$ times the number of ways to write $d$ as a sum of $r$ non-negative integers. The number of ways to do so is a standard calculation; it is equal to $\binom{d+r-1}{r-1}$, which is the same as the multiset number $\left({d+1\choose r-1}\right)$. Solving for $c$ we get $c = \frac{(d+r-1)!}{r!(d-1)!}$.

Putting everything together yields
\begin{align*}
   br(mE) &< (d-1)!\sum_{\substack{a_1 + ... + a_r = d \\ a_1,...,a_r \geq 0}} \bigg(\sum_{i=1}^{r}a_i\lambda(R/I_i) \bigg) \\
   &= (d-1)!\sum_{i=1}^{r}\lambda(R/I_i)c 
   = \frac{(d+r-1)!}{r!} \sum_{i=1}^r \lambda(R/I_i).
\end{align*} Finally, $\lambda(F/E)=\sum_{i=1}^r \lambda(R/I_i)$ yields the result.
\end{proof}

\begin{Th}
Let $(R,m)$ be a Noetherian local ring with dim $R = d \geq 4$, and let $I_1,..,I_d$ be $m$-primary ideals. Then

$$e(mI_1,...,mI_d) < (d-1)! \sum_{i=1}^{d}\lambda(R/I_i)e(R).$$
\end{Th}

\begin{proof} From \ref{reduction-step} we reduce to the case where $R$ is a polynomial ring over an infinite field. The result now follows directly from 4.4.
\end{proof}

\textbf{Acknowledgement:} The authors would like to thank Bernd Ulrich for suggesting the problem and for his insights while we were working on the problem. We are also grateful for his thorough reading of our drafts and for his various corrections.


\bibliographystyle{amsplain}
\bibliography{MultiplicityPaper}

\end{document}